\documentclass[12pt]{article}
\usepackage{graphicx}
\usepackage{amsmath,amsthm,amssymb,amsfonts,euscript,enumerate}
\usepackage{amsthm}
\usepackage{color,wrapfig}
\usepackage{pxfonts}
\usepackage{verbatim}
\newtheorem{thm}{Theorem}[section]
\newtheorem{cor}[thm]{Corollary}
\newtheorem{lem}[thm]{Lemma}
\newtheorem{prop}[thm]{Proposition}

\newtheorem{rmk}[thm]{Remark}

\newcommand{\R}{{\mathbb{R}}}

\newcommand{\1}{\partial}

\newcommand{\3}{\varepsilon}
\newcommand{\4}{\widetilde}

\begin{document}
\title{Global behaviour of  solutions of the fast diffusion equation}
\author{Shu-Yu Hsu\\
Department of Mathematics\\
National Chung Cheng University\\
168 University Road, Min-Hsiung\\
Chia-Yi 621, Taiwan, R.O.C.\\
e-mail: shuyu.sy@gmail.com}
\date{Dec 21, 2017}
\smallbreak \maketitle
\begin{abstract}
We will extend a recent result of B.~Choi and P.~Daskalopoulos (\cite{CD}). For any $n\ge 3$, $0<m<\frac{n-2}{n}$, $m\ne\frac{n-2}{n+2}$, $\beta>0$ and $\lambda>0$, we prove the higher order expansion of the radially symmetric solution $v_{\lambda,\beta}(r)$
of $\frac{n-1}{m}\Delta v^m+\frac{2\beta}{1-m} v+\beta x\cdot\nabla v=0$ in $\R^n$, $v(0)=\lambda$, as $r\to\infty$. As a consequence for any $n\ge 3$ and $0<m<\frac{n-2}{n}$ 
if $u$ is the solution  of the equation $u_t=\frac{n-1}{m}\Delta u^m$ in $\R^n\times (0,\infty)$ with initial value $0\le u_0\in L^{\infty}(\R^n)$ satisfying $u_0(x)^{1-m}=
\frac{2(n-1)(n-2-nm)}{(1-m)\beta |x|^2}\left(\log |x|-\frac{n-2-(n+2)m}{2(n-2-nm)}\log (\log |x|)+K_1+o(1))\right)$ as $|x|\to\infty$ for some constants $\beta>0$ and  $K_1\in\R$, then as $t\to\infty$ the rescaled function $\4{u}(x,t)=e^{\frac{2\beta}{1-m}t}u(e^{\beta t}x,t)$ converges uniformly  on every compact subsets of $\R^n$ to  $v_{\lambda_1,\beta}$ for some constant $\lambda_1>0$.
\end{abstract}

\vskip 0.2truein

Key words: higher order expansion, decay rate, large time behaviour, fast diffusion equation

AMS 2010 Mathematics Subject Classification: Primary 35B40
Secondary 35B20, 35B09 

\vskip 0.2truein
\setcounter{section}{0}

\section{Introduction}
\setcounter{equation}{0}
\setcounter{thm}{0}

Recently there is a lot of interest in the following singular diffusion
equation \cite{A}, \cite{DK}, \cite{P}, \cite{V},
\begin{equation}\label{fde-eqn}
u_t=\frac{n-1}{m}\Delta u^m\quad\mbox{ in }\R^n\times (0,T)
\end{equation}
which arises in the study of many physical models and geometric flows. When $0<m<1$,
\eqref{fde-eqn} is called the fast diffusion equation.  On the other hand as observed by P.~Daskalopoulos, M.~del Pino, J.~King, M.~S\'aez, N.~Sesum, and others \cite{DPKS1}, \cite{DPKS2}, \cite{PS}, the metric $g=u^{\frac{4}{n-2}}dy^2$ satisfies the Yamabe flow \cite{B1}, \cite{B2},
\begin{equation}\label{log1-eqn}
\frac{\1 g}{\1 t}=-Rg
\end{equation}
on $\R^n$, $n\ge 3$, for $0<t<T$, where $R$ is the scalar curvature of the metric $g$, if and only if $u$ satisfies \eqref{fde-eqn} with
\begin{equation*}
m=\frac{n-2}{n+2}.
\end{equation*}
For
\begin{equation}\label{m-range}
n\ge 3\quad\mbox{ and }\quad 0<m<\frac{n-2}{n}, 
\end{equation}
asymptotic behaviour of solution of \eqref{fde-eqn} near finite
extinction time was studied by Galaktionov and L.A.~Peletier \cite{GP}
and P.~Daskalopoulos and N.~Sesum \cite{DS}. Extinction profile
for solutions of \eqref{fde-eqn} for the case $m=\frac{n-2}{n+2}$, $n\ge 3$, was studied
in \cite{PS} by M.~del Pino and M.~S\'aez. Such extinction behaviour occurs for the solution of \eqref{fde-eqn} when \eqref{m-range} holds and the solution $u\in L^{\infty}(\R^n\times (0,T))$ satisfies
\begin{equation*}
u(x,t)\le C|x|^{-\frac{2}{1-m}}\quad\forall |x|\ge R_1, 0\le t<T
\end{equation*}
for some constants $C>0$ and $R_1>0$ (cf. \cite{DS}). On the other hand if \eqref{m-range} holds and $0\le u_0\in L_{loc}^p(\R^n)$ for some constant $p>\frac{n(1-m)}{2}$ which satisfies 
\begin{equation*}
\liminf_{R\to\infty}\frac{1}{R^{n-\frac{2}{1-m}}}\int_{|x|\le R}u_0\,dx
=\infty,
\end{equation*}
then S.Y.~Hsu \cite{Hs2} proved the existence and uniqueness of solutions of 
\begin{equation}\label{cauchy-problem}
\left\{\begin{aligned}
u_t=&\frac{n-1}{m}\Delta u^m\quad\mbox{ in }\R^n\times (0,\infty)\\
u(x,0)=&u_0(x)\qquad\quad\mbox{ in }\R^n.
\end{aligned}\right.
\end{equation}
These results say that we will have either global existence of solution of \eqref{fde-eqn} or extinction in finite time for solution of \eqref{fde-eqn}  depending on whether the growth rate of the solution is large enough at infinity. 
When \eqref{m-range} holds, existence, uniqueness and decay rate of self-similar solutions of \eqref{fde-eqn} were also proved by S.Y.~Hsu in \cite{Hs1}.
Interested reader can read the book
\cite{DK} by P.~Daskalopoulos and C.E.~Kenig and the book \cite{V} by
J.L.~Vazquez for the most recent results on \eqref{fde-eqn}.

In the recent paper
\cite{CD} B.~Choi and P.~Daskalopoulos proved the higher order expansion of the radially symmetric solution $v_{\lambda,\beta}(r)$ of
\begin{equation}\label{elliptic-eqn}
\left\{\begin{aligned}
&\frac{n-1}{m}\Delta v^m+\frac{2\beta}{1-m} v+\beta x\cdot\nabla v=0,\quad v>0,\quad\mbox{ in }\R^n\\
&v(0)=\lambda\end{aligned}\right.
\end{equation}
 for any constant $\lambda>0$ where $m=\frac{n-2}{n+2}$, $n\ge 3$, as $r\to\infty$. They also proved that if $u$ is the solution of \eqref{cauchy-problem} in $\R^n\times (0,\infty)$ with $m=\frac{n-2}{n+2}$, $n\ge 3$,  and initial value $u_0\ge 0$ satisfying 
\begin{equation*}
u_0(x)^{1-m}\approx \frac{(n-1)(n-2)}{\beta |x|^2}\left(\log |x|+K_1+o(1)\right)\quad\mbox{ as }r=|x|\to\infty
\end{equation*}
for some constants $\beta>0$ and $K_1\in\R$, then as $t\to\infty$ the rescaled function 
\begin{equation}\label{rescaled-soln}
\4{u}(x,t)=e^{\frac{2\beta}{1-m}t}u(e^{\beta t}x,t)
\end{equation}
converges uniformly on every compact subsets of $\R^n$ to  $v_{\lambda_1,\beta}(x)$ for some constant $\lambda_1>0$.
Note that for any solution $u$ of \eqref{fde-eqn} in $\R^n\times (0,\infty)$, $\4{u}$ satisfies
\begin{equation}\label{u-tilde-eqn}
\4{u}_t=\frac{n-1}{m}\Delta\4{u}+\frac{2\beta}{1-m}\4{u}+\beta x\cdot\nabla\4{u}\quad\mbox{ in }\R^n\times (0,\infty).
\end{equation}
This result of B.~Choi and P.~Daskalopoulos \cite{CD} shows that the asymptotic large time behaviour of the solution of \eqref{cauchy-problem} depends critically on the higher order expansion of the initial value of the solution. Moreover the asymptotic large time behaviour of the solution of \eqref{cauchy-problem} is after a rescaling similar to the solution of \eqref{elliptic-eqn} with the same higher order expansion. Hence it is important to study the higher order expansion of the solution of \eqref{elliptic-eqn}.

In this paper we will extend the result of \cite{CD}. We will prove the higher order expansion of the radially symmetric solution $v_{\lambda,\beta}(r)$ of \eqref{elliptic-eqn} as $r\to\infty$ for any $n\ge 3$, $0<m<\frac{n-2}{n}$, $m\ne\frac{n-2}{n+2}$, $\beta>0$ and $\lambda>0$. We will also prove that 
when $n\ge 3$, $0<m<\frac{n-2}{n}$, and
\begin{equation*}
u_0(x)^{1-m}\approx \frac{2(n-1)(n-2-nm)}{(1-m)\beta |x|^2}\left(\log |x|-\frac{n-2-(n+2)m}{2(n-2-nm)}\log (\log |x|)+K_1+0(1)\right)
\end{equation*}
as $r=|x|\to\infty$ for some constants $\beta>0$ and $K_1\in\R$, then as $t\to\infty$ the rescaled solution $\4{u}$ given by \eqref{rescaled-soln} will converges uniformly on every compact subsets of $\R^n$ to the radially symmetric solution $v_{\lambda_1,\beta}(x)$ of \eqref{elliptic-eqn} for some constant $\lambda=\lambda_1>0$.

We first start with a definition. For any $0\le u_0\in L_{loc}^1(\R^n)$, we say that a function $u$ is a solution of \eqref{cauchy-problem}
if $u>0$  in $\R^n\times (0,\infty)$  is a classical solution of \eqref{fde-eqn} in $\R^n\times (0,\infty)$ and
\begin{equation*}
\|u(\cdot, t)-u_0\|_{L^1(E)}\to 0\quad\mbox{ as }t\to 0
\end{equation*}
for any compact subset $E$ of $\R^n$. For any $\lambda>0$ and $\beta>0$, we say that $v$ is a solution of \eqref{elliptic-eqn} if $v$ is a positive classical solution of \eqref{elliptic-eqn} in $\R^n$. When there is no ambiguity we will drop the subscript and write $v$ for the radially symmetric solution $v_{\lambda,\beta}$ of \eqref{elliptic-eqn}. Let 
\begin{equation}\label{w-defn}
w(s)=r^2v(r)^{1-m}\quad\mbox{ and }\quad s=\log r.
\end{equation}
We will assume that $n\ge 3$, $0<m<\frac{n-2}{n}$,  $\beta>0$, $\lambda>0$ and $w$ be given by \eqref{w-defn} for the rest of this paper. Unless stated otherwise we will also assume that $m\ne\frac{n-2}{n+2}$.

We obtain the following two main theorems in this paper.

\begin{thm}\label{higher-order-expansion-thm}
Let $n\ge 3$, $0<m<\frac{n-2}{n}$, $m\ne\frac{n-2}{n+2}$, $\beta>0$ and $\lambda>0$. Let $v_{\lambda,\beta}(r)$ be the radially symmetric solution of \eqref{elliptic-eqn} given by \cite{Hs1}. Then there exists a constant $K_0$ independent of $\beta$ and $\lambda$ and a constant $K(\lambda,\beta)$ such that
\begin{align}\label{v-lambda-beta-expansion1}
v_{\lambda,\beta}(r)^{1-m}=&\frac{2(n-1)(n-2-nm)}{(1-m)\beta r^2}\left\{\log r -\frac{n-2-(n+2)m}{2(n-2-nm)}\log (\log r))+\frac{1-m}{2}\log\lambda
\right.\notag\\
&\qquad \left.+\frac{1}{2}\log\beta +K_0+\frac{a_0}{\log r}
+\frac{(n-2-(n+2)m)^2}{4(n-2-nm)^2}\cdot\frac{\log(\log r)}{\log r}+\frac{o(\log r)}{\log r}\right\}
\end{align}
as $r\to\infty$ where
\begin{equation}\label{a0-defn}
a_0=\frac{(n-2-(n+2)m)^2}{4(n-2-nm)^2}-\frac{(1-m)^2 a_1(1,1)}{4(n-1)(n-2-nm)^2}
\end{equation}
and
\begin{align}\label{a1-defn}
a_1(\lambda,\beta)=&\frac{2(1-2m)(n-1)(n-2-nm)}{(1-m)^2}+\frac{(n-1)(n-2-(n+2)m)^2}{(1-m)^2}\notag\\
&\qquad +\frac{(n-2-(n+2)m)}{(1-m)}K(\lambda,\beta)\beta.
\end{align}
\end{thm}

\begin{thm}\label{convergence-thm}
Let $n\ge 3$, $0<m<\frac{n-2}{n}$, $m\ne\frac{n-2}{n+2}$, $\beta>0$, $0\le u_0=\phi+\psi$, $ \phi\in  L_{loc}^p(\R^n)$, $\psi\in L^1(\R^n)\cap L_{loc}^p(\R^n)$, for some constant $p>\frac{(1-m)n}{2}$, be such that
\begin{align}\label{phi-expansion}
\phi(x)^{1-m}=&\frac{2(n-1)(n-2-nm)}{(1-m)\beta |x|^2}\left(\log |x|-\frac{n-2-(n+2)m}{2(n-2-nm)}\log (\log |x|)+K_1+o(1)\right)\mbox{ as }|x|\to\infty
\end{align}
for some constant $K_1\in\R$.
If $u$ is the unique solution of \eqref{fde-eqn} in $\R^n\times (0,\infty)$ given by
Theorem 1.1 of \cite{Hs2}, then as $t\to\infty$, the rescaled function $\4{u}(x,t)$ given by \eqref{rescaled-soln} converges to  $v_{\lambda_1,\beta}$ in $L_{loc}^1(\R^n)$ with $\lambda_1=\left(e^{2K_1/K_0}/\beta\right)^{\frac{1}{1-m}}$ where the constant $K_0$ is given by Theorem \ref{higher-order-expansion-thm}.

Moreover if $u_0$ also satisfies $u_0=\phi\in L^{\infty}(\R^n)$, then as $t\to\infty$,  $\4{u}(x,t)$  also converges to  $v_{\lambda_1,\beta}$ uniformly in $C^{2,1}(E)$ for any compact subset $E\subset\R^n$. 
\end{thm}

\section{Proofs}
\setcounter{equation}{0}
\setcounter{thm}{0}

In this section we will give the proof of Theorem \ref{higher-order-expansion-thm} and Theorem \ref{convergence-thm}. We first recall some results of \cite{Hs1}, \cite{DS} and \cite{HK}.

\begin{thm}\label{w's-bd-thm}(Theorem 1.3  of \cite{Hs1} and its proof)
Let $v$ be the unique radially symmetric solution of \eqref{elliptic-eqn} and $w$ be given by \eqref{w-defn}. Then 
\begin{equation}\label{w-decay-infty}
\lim_{|x|\to\infty}\frac{|x|^2v(x)^{1-m}}{\log |x|}=\lim_{s\to\infty}\frac{w(s)}{s}=\lim_{s\to\infty}w_s(s)
=\frac{2(n-1)(n-2-nm)}{(1-m)\beta}.
\end{equation}
\end{thm}

\begin{lem}(cf. Corollary 2.2 of \cite{DS} and Lemma 2.2 of \cite{HK})\label{L1-comparison-lem}
Let $0\le u_{0,1}, u_{0,2}\in L_{loc}^1(\R^n)$ be such that $u_{0,1}-u_{0,2}\in L^1(\R^n)$. Suppose $u_1$, $u_2$, are solutions of \eqref{fde-eqn} in $\R^n\times (0,\infty)$ with initial values $u_{0,1}, u_{0,2}$ respectively such that for any $T>0$ there exist constants $C_1>0$, $R_1>0$, such that
\begin{equation*}
u_i(x,t)\ge C_1|x|^{-\frac{2}{1-m}}\quad\forall |x|\ge R_1, 0<t<T, i=1,2.
\end{equation*}
Then
\begin{equation*}
\|u_1(\cdot,t)-u_2(\cdot,t)\|_{L^1(\R^n)}\le \|u_{0,1}-u_{0,2}\|_{L^1(\R^n)}\quad\forall t>0
\end{equation*}
and hence
\begin{equation*}
\|\4{u}_1(\cdot,t)-\4{u}_2(\cdot,t)\|_{L^1(\R^n)}\le e^{-\frac{(n-2-nm)}{1-m}t} \|\4{u}_{0,1}-\4{u}_{0,2}\|_{L^1(\R^n)}\quad\forall t >0
\end{equation*}
where $\4{u}_1(\cdot,t)$, $\4{u}_2(\cdot,t)$, are the rescaled solutions of $u_1$, $u_2$, given by \eqref{rescaled-soln}.
\end{lem}

By the computation of \cite{Hs1} we have
\begin{equation}\label{w-eqn}
w_{ss}=\frac{1-2m}{1-m}\cdot\frac{w_s^2}{w}
-\frac{n-2-(n+2)m}{1-m}w_s+\frac{\beta}{n-1}\left(\frac{2(n-1)(n-2-nm)}{(1-m)\beta}-w_s\right)w\quad\mbox{ in }\R.
\end{equation}
Let
\begin{equation}\label{h-defn10}
h(s)=w(s)-\frac{2(n-1)(n-2-nm)}{(1-m)\beta}s.
\end{equation}
Then by \eqref{w-eqn},
\begin{align*}
h_{ss}=&\frac{1-2m}{1-m}\cdot\frac{w_s^2}{w}
-\frac{(n-2-(n+2)m)}{1-m}\left(\frac{2(n-1)(n-2-nm)}{(1-m)\beta} +h_s\right)\notag\\
&\qquad -\frac{\beta}{n-1}\left(\frac{2(n-1)(n-2-nm)}{(1-m)\beta}s+h\right)h_s\quad\mbox{ in }\R.
\end{align*}
Hence
\begin{equation}\label{h-eqn2}
h_{ss}+\left(\frac{2(n-2-nm)}{(1-m)}s+\frac{\beta}{n-1}h+\frac{n-2-(n+2)m}{1-m}\right)h_s=\frac{1-2m}{1-m}\cdot\frac{w_s^2}{w}-b_0\quad\mbox{ in }\R
\end{equation}
where
\begin{equation*}
b_0=\frac{2(n-1)(n-2-nm)(n-2-(n+2)m)}{(1-m)^2\beta}.
\end{equation*}

\begin{lem}\label{h-limit-lem1}
Let $n\ge 3$, $0<m<\frac{n-2}{n}$ and $m\ne\frac{n-2}{n+2}$. Then $h$ satisfies
\begin{equation}\label{h-limit}
\lim_{s\to\infty}\frac{h(s)}{\log s}=\lim_{s\to\infty}s\,h_s(s)=-\frac{(n-1)[n-2-(n+2)m]}{(1-m)\beta}.
\end{equation}
\end{lem}
\begin{proof}
We first observe that by Theorem \ref{w's-bd-thm},
\begin{equation}\label{w-ratio-0}
\lim_{s\to\infty}\frac{w_s^2(s)}{w(s)}=\frac{\lim_{s\to\infty}w_s^2(s)}{\lim_{s\to\infty}s\cdot (w(s)/s)}=0.
\end{equation}
Then by \eqref{h-eqn2} and \eqref{w-ratio-0} for any $0<\3<|b_0|/2$ there exists a constant $s_1\in\R$ such that
\begin{equation}\label{h-ineqn}
-b_0-\3\le h_{ss}+\left(\frac{2(n-2-nm)}{(1-m)}s+\frac{\beta}{n-1}h+\frac{n-2-(n+2)m}{1-m}\right)h_s\le -b_0+\3\quad\forall s\ge s_1.
\end{equation}
Let 
\begin{equation}\label{f-defn}
f(s)=\mbox{exp}\,\left(\frac{(n-2-nm)}{(1-m)}s^2+\frac{\beta}{n-1}\int_{s_1}^sh(z)\,dz+\frac{n-2-(n+2)m}{1-m}s\right).
\end{equation}
By \eqref{h-ineqn},
\begin{align}\label{h'-ineqn}
&(-b_0-\3)f(s)\le (f(s)h_s(s))_s\le (-b_0+\3)f(s)\quad\forall s\ge s_1\notag\\
\Rightarrow\quad&\frac{f(s_1)h_s(s_1)s+(-b_0-\3)s\int_{s_1}^sf(z)\,dz}{f(s)}\le sh_s(s)\le\frac{f(s_1)h_s(s_1)s+(-b_0+\3)s\int_{s_1}^sf(z)\,dz}{f(s)}\,\,\forall s\ge s_1.
\end{align}
By \eqref{w-decay-infty}, $h_s(s)\to 0$ as $s\to\infty$. Hence $h(s)=o(s)$ and $h(s)/s\to 0$ as $s\to\infty$. Then by \eqref{f-defn}
$f(s)\to\infty$ as $s\to\infty$. Hence by the l'Hospital rule,
\begin{equation}\label{f-expression-ratio-limit}
\lim_{s\to\infty}\frac{\int_{s_1}^sf(z)\,dz}{f(s)}=\lim_{s\to\infty}\frac{f(s)}{f(s)\left(\frac{2(n-2-nm)}{(1-m)}s+\frac{\beta}{n-1}h(s)+\frac{n-2-(n+2)m}{1-m}\right)}=0
\end{equation}
and
\begin{equation}\label{ratio-limit}
\lim_{s\to\infty}\frac{s}{f(s)}=\lim_{s\to\infty}\frac{1}{f(s)\left(\frac{2(n-2-nm)}{(1-m)}s+\frac{\beta}{n-1}h(s)+\frac{n-2-(n+2)m}{1-m}\right)}=0.
\end{equation}
Hence by \eqref{f-expression-ratio-limit} and the l'Hospital rule,
\begin{align}\label{ratio-limit2}
\lim_{s\to\infty}\frac{s\int_{s_1}^sf(z)\,dz}{f(s)}
=&\lim_{s\to\infty}\frac{sf(s)+\int_{s_1}^sf(z)\,dz}{f(s)\left(\frac{2(n-2-nm)}{(1-m)}s+\frac{\beta}{n-1}h(s)+\frac{n-2-(n+2)m}{1-m}\right)}\notag\\
=&\frac{(1-m)}{2(n-2-nm)}.
\end{align}
Letting first $s\to\infty$ and then $\3\to 0$ in \eqref{h'-ineqn}, by \eqref{ratio-limit} and  \eqref{ratio-limit2} we have
\begin{equation}\label{h'-limit10}
\lim_{s\to\infty}s\, h_s(s)=-\frac{(1-m)b_0}{2(n-2-nm)}=-\frac{(n-1)[n-2-(n+2)m]}{(1-m)\beta}.
\end{equation}
Hence 
\begin{equation}\label{h-ineqn10}
h(s)\le -\frac{(n-1)[n-2-(n+2)m]}{2(1-m)\beta}\log s\quad\mbox{ as }s\to\infty
\end{equation}
and \eqref{h-limit} follows from  \eqref{h'-limit10}, \eqref{h-ineqn10}   and the l'Hospital rule.
\end{proof} 

\begin{cor}
Let $n\ge 3$, $0<m<\frac{n-2}{n}$ and $m\ne\frac{n-2}{n+2}$. Then there exists a constant $s_0\in\R$ such that
\begin{equation*}
\left\{\begin{aligned}
&h_s(s)<0\quad\forall s\ge s_0\qquad\mbox{ if }\,\,0<m<\frac{n-2}{n+2}\\
&h_s(s)>0\quad\forall s\ge s_0\qquad\mbox{ if }\,\,\frac{n-2}{n+2}<m<\frac{n-2}{n}.
\end{aligned}\right.
\end{equation*}
\end{cor}

Let 
\begin{equation}\label{h1-defn}
h_1(s)=h(s)+\frac{(n-1)[n-2-(n+2)m]}{(1-m)\beta}\log s.
\end{equation}
Then by Lemma \ref{h-limit-lem1}, $h_1(s)=o(\log s)$ as $s\to\infty$. By \eqref{h-eqn2} $h_1$ satisfies
\begin{align}\label{h1-eqn2}
&h_{1,ss}+\left(\frac{2(n-2-nm)}{(1-m)}s+\frac{\beta}{n-1}h+\frac{n-2-(n+2)m}{1-m}\right)h_{1,s}\notag\\
=&\frac{1-2m}{1-m}\cdot\frac{w_s^2}{w}+a_2\left[-\frac{1}{s^2}+\left(\frac{\beta}{n-1}h+\frac{n-2-(n+2)m}{1-m}\right)\frac{1}{s}\right]\quad\mbox{ in }\R
\end{align}
where
\begin{equation*}
a_2=\frac{(n-1)[n-2-(n+2)m]}{(1-m)\beta}.
\end{equation*}

\begin{lem}
Let $0<m<\frac{n-2}{n}$, $m\ne\frac{n-2}{n+2}$ , $\lambda>0$ and $\beta>0$. Then 
\begin{equation}\label{h1'-limit}
\lim_{s\to\infty}\frac{s^2h_{1,s}(s)}{\log s}=-\frac{(n-1)(n-2-(n+2)m)^2}{2(n-2-nm)(1-m)\beta}.
\end{equation}
\end{lem}
\begin{proof}
Let 
\begin{equation*}
H(s)=\frac{1-2m}{1-m}\cdot\frac{w_s^2}{w}+a_2\left[-\frac{1}{s^2}+\left(\frac{\beta}{n-1}h+\frac{n-2-(n+2)m}{1-m}\right)\frac{1}{s}\right].
\end{equation*}
Then by  Theorem \ref{w's-bd-thm} and Lemma \ref{h-limit-lem1},
\begin{equation}\label{H-infty-limit}
\lim_{s\to\infty}\frac{sH(s)}{\log s}=-\frac{(n-2-(n+2)m)}{1-m}a_2=-a_3
\end{equation}
where
\begin{equation}
a_3=\frac{(n-1)(n-2-(n+2)m)^2}{(1-m)^2\beta}.
\end{equation}
By \eqref{h1-eqn2} and \eqref{H-infty-limit} for any $0<\3<a_3/2$ there exists a constant $s_1>1$ such that
\begin{align}\label{h1-ineqn}
& (-a_3-\3)\frac{\log s}{s}\le h_{1,ss}+\left(\frac{2(n-2-nm)}{(1-m)}s+\frac{\beta}{n-1}h+\frac{n-2-(n+2)m}{1-m}\right)h_{1,s}\le (-a_3+\3)\frac{\log s}{s}
\end{align}
holds for all $s\ge s_1$. Let $f$ be given by \eqref{f-defn}. Multiplying \eqref{h1-ineqn} by $f$ and integrating over $(s_1,s)$,
\begin{align}\label{h1'-ineqn3}
\frac{s^2}{\log s}\left(\frac{f(s_1)h_{1,s}(s_1)+(-a_3-\3)\int_{s_1}^s\frac{\log z}{z}f(z)\,dz}{f(s)}\right)\le& \frac{s^2h_{1,s}(s)}{\log s}\notag\\
\le&
\frac{s^2}{\log s}\left(\frac{f(s_1)h_{1,s}(s_1)+(-a_3+\3)\int_{s_1}^s\frac{\log z}{z}f(z)\,dz}{f(s)}\right)
\end{align}
holds for all $s\ge s_1$.
By the l'Hospital rule and Lemma \ref{h-limit-lem1},
\begin{align}\label{f-expression-ratio-limit2}
\lim_{s\to\infty}\frac{s^2\int_{s_1}^s\frac{\log z}{z}f(z)\,dz}{ f(s)\log s}
=&\lim_{s\to\infty}\frac{f(s)s\log s+2s\int_{s_1}^s\frac{\log z}{z}f(z)\,dz}{f(s)\left(\frac{2(n-2-nm)}{(1-m)}s+\frac{\beta}{n-1}h(s)+\frac{n-2-(n+2)m}{1-m}\right)\log s+\frac{f(s)}{s}}\notag\\
=&\frac{1-m}{2(n-2-nm)}+\frac{1-m}{n-2-nm}\lim_{s\to\infty}\frac{\int_{s_1}^s\frac{\log z}{z}f(z)\,dz}{f(s)\log s}
\end{align}
Since 
\begin{equation*}
\left|\frac{\int_{s_1}^s\frac{\log z}{z}f(z)\,dz}{f(s)\log s}\right|\le\frac{\int_{s_1}^sf(z)\,dz}{f(s)}\quad\forall s\ge s_1,
\end{equation*}
by \eqref{f-expression-ratio-limit},
\begin{equation}\label{f-expression-ratio-limit3}
\lim_{s\to\infty}\frac{\int_{s_1}^s\frac{\log z}{z}f(z)\,dz}{f(s)\log s}=0.
\end{equation}
By \eqref{f-expression-ratio-limit2} and \eqref{f-expression-ratio-limit3},
\begin{equation}\label{f-expression-ratio-limit4}
\lim_{s\to\infty}\frac{s^2\int_{s_1}^s\frac{\log z}{z}f(z)\,dz}{ f(s)\log s}=\frac{1-m}{2(n-2-nm)}.
\end{equation}
Letting first $s\to\infty$ and then $\3\to 0$ in \eqref{h1'-ineqn3}, by \eqref{f-expression-ratio-limit4} we get \eqref{h1'-limit}
and the lemma follows.
\end{proof}

\begin{cor}\label{h1-expansion-cor}
Let $n\ge 3$, $0<m<\frac{n-2}{n}$, $m\ne\frac{n-2}{n+2}$, $\lambda>0$ and $\beta>0$. Then 
\begin{equation}\label{K-defn}
K(\lambda,\beta):=\lim_{s\to\infty}h_1(s)\in\R\quad\mbox{ exists}
\end{equation} 
and
\begin{equation}\label{h1-expression}
h_1(s)=K(\lambda,\beta)+\frac{(n-1)(n-2-(n+2)m)^2}{2(n-2-nm)(1-m)\beta}\left(\frac{1+\log s}{s}\right)+o\left(\frac{1+\log s}{s}\right)\quad\mbox{ as }s\to\infty.
\end{equation}
\end{cor}
\begin{proof}
By \eqref{h1'-limit} there exist constants $C_1>0$ and $s_1>1$ such that
\begin{equation}
\left|\frac{s^2h_{1,s}(s)}{\log s}\right|\le C_1\quad\forall s\ge s_1.
\end{equation} 
Hence
\begin{align}\label{h1-uniform-bd}
&|h_1(s)-h_1(s_1)|\le\int_{s_1}^s|h_{1,s}(z)|\,dz\le C_1\int_{s_1}^s\frac{\log z}{z^2}\,dz\le C_2\quad\forall s\ge s_1\notag\\
\Rightarrow\quad&|h_1(s)|\le C_2+|h_1(s_1)|\quad\forall s\ge s_1
\end{align}
for some constant $C_2>0$. On the other hand by \eqref{h1'-limit} there exists a constant $s_0>s_1$ such that
\begin{equation*}
h_{1,s}(s)<0\quad\forall s\ge s_0.
\end{equation*}
Then $h_1(s)$ is  monotone decreasing in $(s_0,\infty)$. This together with \eqref{h1-uniform-bd} implies that \eqref{K-defn} holds.
By \eqref{h1'-limit} for any 
$0<\3<\frac{(n-1)(n-2-(n+2)m)^2}{4(n-2-nm)(1-m)\beta}$ there exists a constant $s_2>1$ such that
\begin{equation}\label{h1'-ineqn6}
\left(-\frac{(n-1)(n-2-(n+2)m)^2}{2(n-2-nm)(1-m)\beta}-\3\right)\frac{\log s}{s^2}\le h_{1,s}(s)\le \left(-\frac{(n-1)(n-2-(n+2)m)^2}{2(n-2-nm)(1-m)\beta}+\3\right)\frac{\log s}{s^2}
\end{equation}
holds for all $s\ge s_2$.
Integrating \eqref{h1'-ineqn6} over $(s,\infty)$, $s\ge s_2$,
\begin{align*}
-\left(\frac{(n-1)(n-2-(n+2)m)^2}{2(n-2-nm)(1-m)\beta}+\3\right)\frac{(1+\log s)}{s}
\le &K(\lambda,\beta)-h_1(s)\\
\le&-\left(\frac{(n-1)(n-2-(n+2)m)^2}{2(n-2-nm)(1-m)\beta}-\3\right)\frac{(1+\log s)}{s}
\end{align*}
for all $s\ge s_2$ and \eqref{h1-expression} follows.
\end{proof}

Let $K(\lambda,\beta)$ be given by \eqref{K-defn} and
\begin{equation}\label{h2-defn}
h_2(s)=h_1(s)-K(\lambda,\beta)-\frac{(n-1)(n-2-(n+2)m)^2}{2(n-2-nm)(1-m)\beta}\left(\frac{1+\log s}{s}\right).
\end{equation}
Then 
\begin{equation}\label{h2-infty}
h_2(s)=o\left(\frac{1+\log s}{s}\right)\mbox{ as }s\to\infty
\end{equation}
 and by \eqref{h1-defn} and \eqref{h1-eqn2},
\begin{align}\label{h2-eqn}
&h_{2,ss}+\left(\frac{2(n-2-nm)}{(1-m)}s+\frac{\beta}{n-1}h+\frac{n-2-(n+2)m}{1-m}\right)h_{2,s}\notag\\
=&\frac{1-2m}{1-m}\cdot\frac{w_s^2}{w}+\frac{(n-1)(n-2-(n+2)m)^2}{(1-m)^2\beta}\cdot\frac{1}{s}+\frac{(n-2-(n+2)m)}{1-m}\cdot\frac{h_1(s)}{s}-\frac{a_2}{s^2}\notag\\
&\qquad +\frac{(n-1)(n-2-(n+2)m)^2}{2(1-m)(n-2-nm)\beta}\cdot\left[\frac{(1-2\log s)}{s^3}+\left(\frac{\beta}{n-1}h(s)+\frac{n-2-(n+2)m}{1-m}\right)\cdot\frac{\log s}{s^2}\right]\notag\\
=&:H_1(s).
\end{align}

\begin{lem}\label{h2-limit-lem}
Let $n\ge 3$, $0<m<\frac{n-2}{n}$, $m\ne\frac{n-2}{n+2}$, $\lambda>0$ and $\beta>0$. Then $h_2$ satisfies
\begin{equation}\label{h2-limit}
\lim_{s\to\infty}s^2\,h_{2,s}(s)=\frac{(1-m)a_1(\lambda,\beta)}{2(n-2-nm)\beta}
\end{equation}
where $a_1(\lambda,\beta)$ is given by \eqref{a1-defn} with $K(\lambda,\beta)$ given by \eqref{K-defn}.
\end{lem}
\begin{proof}
By Theorem \ref{w's-bd-thm}, Lemma \ref{h-limit-lem1} and \eqref{K-defn},
\begin{equation}\label{H1-decay-rate}
\lim_{s\to\infty}sH_1(s)=\frac{a_1(\lambda,\beta)}{\beta}
\end{equation}
where $a_1(\lambda,\beta)$ is given by \eqref{a1-defn} with $K(\lambda,\beta)$ given by \eqref{K-defn}.
Then by \eqref{h2-eqn} and \eqref{H1-decay-rate} for any $0<\3<1$ there exists a constant $s_1>1$ such that
\begin{equation}\label{h2-ineqn20}
\left(\frac{a_1(\lambda,\beta)}{\beta}-\3\right)\frac{1}{s}\le h_{2,ss}+\left(\frac{2(n-2-nm)}{(1-m)}s+\frac{\beta}{n-1}h+\frac{n-2-(n+2)m}{1-m}\right)h_{2,s}\le\left(\frac{a_1(\lambda,\beta)}{\beta}+\3\right)\frac{1}{s} 
\end{equation}
holds for all $s\ge s_1$.
Let $f$ be given by \eqref{f-defn}. Multiplying \eqref{h2-ineqn20} by $f$ and integrating over $(s_1,s)$,
\begin{equation}\label{h-ineqn32}
\frac{f(s_1)h_2(s_1)s^2+\left(\frac{a_1(\lambda,\beta)}{\beta}-\3\right)s^2\int_{s_1}^s\frac{f(z)}{z}\,dz}{f(s)}\le s^2h_{2,s}(s)\le \frac{f(s_1)h_2(s_1)s^2+\left(\frac{a_1(\lambda,\beta)}{\beta}+\3\right)s^2\int_{s_1}^s\frac{f(z)}{z}\,dz}{f(s)}
\end{equation}
holds for all $s\ge s_1$.
Since by the l'Hospital rule,
\begin{equation*}
\lim_{s\to\infty}\frac{\int_{s_1}^s\frac{f(z)}{z}\,dz}{f(s)}=\lim_{s\to\infty}\frac{\frac{f(s)}{s}}{f(s)\left(\frac{2(n-2-nm)}{(1-m)}s+\frac{\beta}{n-1}h(s)+\frac{n-2-(n+2)m}{1-m}\right)}=0,
\end{equation*}
by Lemma \ref{h-limit-lem1} and the l'Hospital rule,
\begin{align}\label{f-ratio-limit}
\lim_{s\to\infty}\frac{s^2\int_{s_1}^s\frac{f(z)}{z}\,dz}{f(s)}=&\lim_{s\to\infty}\frac{sf(s)+2s\int_{s_1}^s\frac{f(z)}{z}\,dz}{f(s)\left(\frac{2(n-2-nm)}{(1-m)}s+\frac{\beta}{n-1}h(s)+\frac{n-2-(n+2)m}{1-m}\right)}\notag\\
=&\frac{(1-m)}{2(n-2-nm)}+\frac{(1-m)}{(n-2-nm)}\lim_{s\to\infty}\frac{\int_{s_1}^s\frac{f(z)}{z}\,dz}{f(s)}\notag\\
=&\frac{(1-m)}{2(n-2-nm)}.
\end{align}
Hence letting first $s\to\infty$ and then $\3\to 0$ in \eqref{h-ineqn32}, by \eqref{f-defn} and   \eqref{f-ratio-limit} we get \eqref{h2-limit} and the lemma follows.
\end{proof}

By \eqref{h2-infty} and Lemma \ref{h2-limit-lem} we have the following result.

\begin{lem}\label{h2-expansion-lem}
Let $n\ge 3$, $0<m<\frac{n-2}{n}$, $m\ne\frac{n-2}{n+2}$, $\lambda>0$ and $\beta>0$.  Then 
\begin{equation}\label{h2-expansion}
h_2(s)=-\frac{(1-m)a_1(\lambda,\beta)}{2(n-2-nm)\beta s}+\frac{o(s)}{s}\quad\mbox{ as }s\to\infty
\end{equation}
where $a_1(\lambda,\beta)$ is given by \eqref{a1-defn} with $K(\lambda,\beta)$ given by \eqref{K-defn}.
\end{lem}

We are now ready for the proof of Theorem \ref{higher-order-expansion-thm}.

\noindent{\bf Proof of Theorem \ref{higher-order-expansion-thm}}:  Let $a_0$, $a_1(\lambda,\beta)$, be given by \eqref{a0-defn} and \eqref{a1-defn} with $K(\lambda,\beta)$ given by \eqref{K-defn}. Let 
\begin{equation*}
K_0=\frac{(1-m)K(1,1)}{2(n-1)(n-2-nm)}.
\end{equation*}
By \eqref{w-defn}, \eqref{h-defn10}, \eqref{h1-defn}, \eqref{h2-defn} and \eqref{h2-expansion}, 
\begin{align}\label{v11-expansion}
v_{1,1}(r)^{1-m}=&\frac{2(n-1)(n-2-nm)}{(1-m)r^2}\left\{\log r-\frac{(n-2-(n+2)m)}{2(n-2-nm)}\log (\log r)+K_0\right.\notag\\
&\qquad+\frac{(n-2-(n+2)m)^2}{4(n-2-nm)^2}\left.\left(\frac{1+\log(\log r)}{\log r}\right)
-\frac{(1-m)^2}{4(n-1)(n-2-nm)^2}\frac{a_1(1,1)}{\log r}\right.\notag\\
&\qquad \left.+\frac{o(\log r)}{\log r}\right\}\notag\\
=&\frac{2(n-1)(n-2-nm)}{(1-m)r^2}\left\{\log r-\frac{(n-2-(n+2)m)}{2(n-2-nm)}\log (\log r)+K_0+\frac{a_0}{\log r}\right.\notag\\
&\qquad\left.+\frac{(n-2-(n+2)m)^2}{4(n-2-nm)^2}\cdot\frac{\log(\log r)}{\log r}+\frac{o(\log r)}{\log r}\right\}\quad\mbox{ as }r\to\infty.
\end{align}
Then by (2.19) of \cite{CD} and \eqref{v11-expansion},
\begin{align*}
v_{\lambda,\beta}(r)^{1-m}=&\lambda^{1-m} v_{1,1}(\lambda^{\frac{1-m}{2}}\sqrt{\beta}r)^{1-m}\notag\\
=&\frac{2(n-1)(n-2-nm)}{(1-m)\beta r^2}\left\{\log (\lambda^{\frac{1-m}{2}}\sqrt{\beta}r)-\frac{(n-2-(n+2)m)}{2(n-2-nm)}\log (\log (\lambda^{\frac{1-m}{2}}\sqrt{\beta}r))+K_0
\right.\notag\\
&\qquad+\frac{a_0}{\log (\lambda^{\frac{1-m}{2}}\sqrt{\beta}r)}
+\frac{(n-2-(n+2)m)^2}{4(n-2-nm)^2}\cdot\frac{\log(\log (\lambda^{\frac{1-m}{2}}\sqrt{\beta}r))}{\log (\lambda^{\frac{1-m}{2}}\sqrt{\beta}r)}\notag\\
&\qquad\left.+\frac{o(\log (\lambda^{\frac{1-m}{2}}\sqrt{\beta}r))}{\log (\lambda^{\frac{1-m}{2}}\sqrt{\beta}r)}\right\}\notag\\
=&\frac{2(n-1)(n-2-nm)}{(1-m)\beta r^2}\left\{\log r -\frac{(n-2-(n+2)m)}{2(n-2-nm)}\log (\log r))
+\frac{1-m}{2}\log\lambda\right.\notag\\
&\qquad \left.+\frac{1}{2}\log\beta+K_0+\frac{a_0}{\log r}
+\frac{(n-2-(n+2)m)^2}{4(n-2-nm)^2}\cdot\frac{\log(\log r)}{\log r}+\frac{o(\log r)}{\log r}\right\}
\end{align*}
as $r\to\infty$ and Theorem \ref{higher-order-expansion-thm} follows. 

{\hfill$\square$\vspace{6pt}} 

\begin{rmk}
From \eqref{v-lambda-beta-expansion1},
\begin{align*}
h_1(s)=&\frac{2(n-1)(n-2-nm)}{(1-m)\beta}\left\{
\frac{1-m}{2}\log\lambda +\frac{1}{2}\log\beta+K_0+\frac{a_0}{s}
+\frac{(n-2-(n+2)m)^2}{4(n-2-nm)^2}\cdot\frac{\log s}{s}\right.\\
&\qquad\left.+\frac{o(s)}{s}\right\}\qquad\mbox{ as }s\to\infty.
\end{align*}
Hence 
\begin{equation}\label{h1-limit2}
\lim_{s\to\infty}h_1(s)=\frac{2(n-1)(n-2-nm)}{(1-m)\beta}\left\{
\frac{1-m}{2}\log\lambda +\frac{1}{2}\log\beta+K_0\right\}.
\end{equation}
Thus by \eqref{K-defn} and \eqref{h1-limit2},
\begin{equation*}
K(\lambda,\beta)=\frac{2(n-1)(n-2-nm)}{(1-m)\beta}\left\{
\frac{1-m}{2}\log\lambda +\frac{1}{2}\log\beta+K_0\right\}.
\end{equation*}
\end{rmk}
 
\noindent{\bf Proof of Theorem \ref{convergence-thm}}: Since the proof is similar to the proof of Theorem 3.1 and Corollary 3.2 of \cite{CD} we will only sketch the proof here. Let $K_0$ be given by Theorem \ref{higher-order-expansion-thm}  and $\lambda_1=\left(e^{2(K_1-K_0)}/\beta\right)^{\frac{1}{1-m}}$. Then for any $0<\3<\lambda_1$ there exists a constant $R_{\3}>0$ such that
\begin{align}\label{initial-data-compare}
&u_{\lambda_1-\3,\beta}(x)\le \phi(x)\le u_{\lambda_1+\3,\beta}(x)\quad\forall |x|\ge R_{\3}\notag\\
\Rightarrow\quad&u_{\lambda_1-\3,\beta}(x)+\psi(x)\le u_0(x)\le u_{\lambda_1+\3,\beta}(x)+\psi(x)\quad\forall |x|\ge R_{\3}.
\end{align}
For any $\delta>0$, let $u_1$, $u_2$,  $u_{1,\delta}$ and $w_{1,\delta}$ be the solution of \eqref{cauchy-problem} with initial value 
\begin{equation*}
\min (u_{\lambda_1-\3,\beta}(x)+\psi(x), u_0(x)),\,\, \max (u_{\lambda_1+\3,\beta}(x)+\psi(x), u_0(x)),\,\, \min (u_{\lambda_1-\3,\beta}(x)+\psi(x), u_0(x))+\delta,
\end{equation*}
 and $u_{\lambda_1-\3,\beta}(x)+\delta$
respectively given by Theorem 1.1 of \cite{Hs2}. Let $\4{u}_1$, $\4{u}_2$, $\4{u}_{1,\delta}$ and  $\4{w}_{1,\delta}$ be given by \eqref{rescaled-soln} with $u$ being replaced by $u_1$, $u_2$, $u_{1,\delta}$ and $w_{1,\delta}$ respectively. By \eqref{initial-data-compare}
and the construction of solutions in \cite{Hs2},
\begin{align}
&u_1\le u\le u_2, \quad u_{1,\delta}\ge\delta, \quad w_{1,\delta}\ge\delta\quad\mbox{ in }\R^n\times (0,\infty)\quad\forall\delta>0\label{u1-u-u2-compare}\\
\Rightarrow\quad& \4{u}_1\le \4{u}\le\4{u}_2, \quad\mbox{ in }\R^n\times (0,\infty)\label{u1-u-u2-tilde-compare2}.
\end{align}
By \eqref{u1-u-u2-compare} and Lemma \ref{L1-comparison-lem},
\begin{equation}\label{u-tidle-L1-comparison}
\|\4{u}_{1,\delta}(\cdot,t)-\4{w}_{1,\delta}(\cdot,t)\|_{L^1(\R^n)}\le e^{-\frac{(n-2-nm)}{1-m}t} \|\min (u_{\lambda_1-\3,\beta}+\psi, u_0)-u_{\lambda_1-\3,\beta}\|_{L^1(\R^n)}\quad\forall t >0
\end{equation} 
Since $\min (u_{\lambda_1-\3,\beta}(x)+\psi(x), u_0(x))+\delta$ and $u_{\lambda_1-\3,\beta}(x)+\delta$ decreases monotonically to 
\begin{equation*}
\min (u_{\lambda-\3,\beta}(x)+\psi(x), u_0(x))\quad\mbox{ and } 
\quad u_{\lambda_1-\3,\beta}(x)
\end{equation*}
 as $\delta\to 0$, $u_{1,\delta}$ and $w_{1,\delta}$ decreases monotonically to $u_1$ and 
$e^{-\frac{2\beta}{1-m}t}u_{\lambda_1-\3,\beta}(e^{-\beta t}x)$ as $\delta\to 0$.
Hence letting $\delta\to 0$ in \eqref{u-tidle-L1-comparison},
\begin{equation}\label{u-tidle-L1-comparison2}
\|\4{u}_1(\cdot,t)-u_{\lambda_1-\3,\beta}\|_{L^1(\R^n)}\le e^{-\frac{(n-2-nm)}{1-m}t} \|\min (u_{\lambda_1-\3,\beta}+\psi, u_0)-u_{\lambda_1-\3,\beta}\|_{L^1(\R^n)}\quad\forall t >0
\end{equation} 
Similarly,
\begin{equation}\label{u-tidle-L1-comparison3}
\|\4{u}_2(\cdot,t)-u_{\lambda_1+\3,\beta}\|_{L^1(\R^n)}\le e^{-\frac{(n-2-nm)}{1-m}t} \|\max (u_{\lambda_1-\3,\beta}+\psi, u_0)-u_{\lambda_1+\3,\beta}\|_{L^1(\R^n)}\quad\forall t >0
\end{equation} 
By \eqref{u1-u-u2-tilde-compare2}, \eqref{u-tidle-L1-comparison2} and \eqref{u-tidle-L1-comparison3}, and an argument similar to the proof of Theorem 3.1 of \cite{CD}, the rescaled function $\4{u}(x,t)$ given by \eqref{rescaled-soln} converges to  $v_{\lambda_1,\beta}$ in $L_{loc}^1(\R^n)$ as $t\to\infty$.
 
Suppose now $u_0$ also satisfies $u_0=\phi\in L^{\infty}(\R^n)$. Then by an argument similar to the proof of Corollary 3.2 of \cite{CD}, there exists a constant $\lambda_2>0$ such that
\begin{equation*}
u_0\le v_{\lambda_2,\beta}\quad\mbox{ in }\R^n
\end{equation*}
Hence by maximum principle for solutions of \eqref{fde-eqn} in bounded domains (cf. Lemma 2.3 of \cite{DaK}) and the construction of solution \eqref{cauchy-problem} in \cite{Hs2},
\begin{align}\label{u-tilde-v-bd}
&u(x,t)\le e^{-\frac{2\beta}{1-m}t}v_{\lambda_2,\beta}(e^{-\beta t}x)\quad\forall x\in\R^n, t>0\notag\\
\Rightarrow\quad&\4{u}(x,t)\le v_{\lambda_2,\beta}(x)\qquad\qquad\,\forall x\in\R^n, t>0.
\end{align}  
Then by \eqref{u-tilde-eqn},  \eqref{u-tilde-v-bd}, and an argument similar to the proof on P.10 of \cite{CD}, the rescaled function $\4{u}(x,t)$ converges to  $v_{\lambda_1,\beta}$ uniformly in $C^{2,1}(E)$ for any compact subset $E\subset\R^n$ as $t\to\infty$. 

{\hfill$\square$\vspace{6pt}}  

Finally by Theorem \ref{higher-order-expansion-thm} and a similar argument as the proof of Theorem 3.6 and Proposition 3.9 of \cite{CD} we have the following result.

\begin{thm}\label{convergence-thm2}
Let $n\ge 3$, $0<m<\frac{n-2}{n}$, $m\ne\frac{n-2}{n+2}$, $\beta>0$, $0\le u_0=\phi+\psi$, $ \phi\in  L_{loc}^p(\R^n)$, $\psi\in L^1(\R^n)\cap L_{loc}^p(\R^n)$, for some constant $p>\frac{(1-m)n}{2}$, such that 
\begin{equation*}
K_2:=\limsup_{|x|\to\infty}\left[|x|^2u_0(x)^{1-m}-\frac{c_1}{\beta}\left(\log|x|-\frac{(n-2-(n+2)m)}{2(n-2-nm)}\log(\log |x|)\right)\right]<\infty.
\end{equation*}
holds and $\phi$ satisfies \eqref{phi-expansion} for some constant $K_1\in\R$ where $c_1=2(n-1)(n-2-nm)/(1-m)$.
If $u$ is the unique solution of \eqref{fde-eqn} in $\R^n\times (0,\infty)$ given by
Theorem 1.1 of \cite{Hs2}, then as $t\to\infty$, the rescaled function $\4{u}(x,t)$ given by \eqref{rescaled-soln} converges to  $v_{\lambda_1,\beta}$ uniformly in $C^{2,1}(E)$ for any compact subset $E\subset\R^n$ with $\lambda_1=\left(e^{2(K_1-K_0)}/\beta\right)^{\frac{1}{1-m}}$ where the constant $K_0$ is given by Theorem \ref{higher-order-expansion-thm}. Moreover
\begin{align*}
&\limsup_{|x|\to\infty}\left[|x|^2u(x,t)^{1-m}-\frac{c_1}{\beta}\left(\log|x|-\frac{n-2-(n+2)m}{2(n-2-nm)}\log(\log |x|)\right)\right]\notag\\
\le&K_2-\frac{2(n-1)(n-2-nm)}{(1-m)}t\quad\forall t\ge 0.
\end{align*}
\end{thm}
 
\section{Appendix}
\setcounter{equation}{0}
\setcounter{thm}{0} 

For the sake of completeness, in this appendix we will state and prove the analogue of Lemma \ref{h-limit-lem1} for the case $m=\frac{n-2}{n+2}$, $n\ge 3$.

\begin{prop}(Proposition 2.3 of \cite{CD})
Let $n\ge 3$, $m=\frac{n-2}{n+2}$, $\lambda>0$ and $\beta>0$. Let $v$ be the solution of \eqref{elliptic-eqn} and $h$ be given by  \eqref{h-defn10} with $w$ given by \eqref{w-defn}. Then $h$ satisfies
\begin{equation*}
\lim_{s\to\infty}s^2\,h_s(s)=\frac{(6-n)(n-1)}{4\beta}.
\end{equation*}
\end{prop}
\begin{proof}
This proposition is stated and proved in \cite{CD}. For the sake of completeness we will give a simple different proof of the proposition here. We first observe that by Theorem \ref{w's-bd-thm},
\begin{equation}\label{w-ratio-limit}
\lim_{s\to\infty}\frac{sw_s(s)^2}{w(s)}=\frac{2(n-1)(n-2-nm)}{(1-m)\beta}=\frac{(n-1)(n-2)}{\beta}.
\end{equation}
Let
\begin{equation*}
a_4=\frac{(1-2m)(n-1)(n-2)}{(1-m)\beta}.
\end{equation*}
Then by \eqref{h-eqn2} and \eqref{w-ratio-limit}   for any $0<\3<|a_4|/2$ there exists a constant $s_1\in\R$ such that
\begin{equation}\label{h-ineqn31}
(a_4-\3)s^{-1}\le h_{ss}+\left(\frac{2(n-2-nm)}{(1-m)}s+\frac{\beta}{n-1}h+\frac{n-2-(n+2)m}{1-m}\right)h_s\le (a_4+\3)s^{-1}\quad\forall s\ge s_1.
\end{equation}
By \eqref{h-ineqn31} and an argument similar to the proof of Lemma \ref{h2-limit-lem},
\begin{equation*}
\lim_{s\to\infty}s^2\,h_s(s)=\frac{(1-m)a_4}{2(n-2-nm)}=\frac{(6-n)(n-1)}{4\beta}
\end{equation*}
and the proposition follows.
\end{proof}


\begin{thebibliography}{99}

\bibitem[A]{A} D.G.~Aronson, {\em The porous medium equation}, CIME
Lectures in Some problems in Nonlinear Diffusion, Lecture
Notes in Mathematics 1224, Springer-Verlag, New York, 1986.

\bibitem[B1]{B1} S.~Brendle, {\em Convergence of the Yamabe flow for arbitrary energy}, J. Differential Geom. 69 (2005), 217--278.

\bibitem[B2]{B2} S.~Brendle, {\em Convergence of the Yamabe flow in dimension 6 and higher}, Invent. Math. 170 (2007), 541--576.

\bibitem[CD]{CD} B.~Choi and P.~Daskalopoulos, {\em Yamabe flow: steady solutions and type II singularities}, https://arxiv.org/abs/1709.03192v1.

\bibitem[DaK]{DaK} B.E.J.~Dahlberg and C.~Kenig, {\em Non-negative solutions of generalized porous medium equations},
Revista Matem\'atica Iberoamericana 2 (1986), 267--305.

\bibitem[DK]{DK} P.~Daskalopoulos and C.E.~Kenig, {\em Degenerate
diffusion-initial value problems and local regularity theory},
Tracts in Mathematics 1, European Mathematical Society, 2007.

\bibitem[DPKS1]{DPKS1} P.~Daskalopoulos, M.~del Pino, J.~King and N.~Sesum, {\em Type I ancient compact solutions of the Yamabe flow}, Nonlinear Analysis, Theory, Methods and Applications 137 (2016), 338--356.

\bibitem[DPKS2]{DPKS2} P.~Daskalopoulos, M.~del Pino, J.~King and N.~Sesum, {\em New type I ancient compact solutions of the Yamabe flow}, arXiv:1601.05349v1.


\bibitem[DS]{DS} P.~Daskalopoulos and N.~Sesum, {\em On the extinction profile
of solutions to fast diffusion,} J. Reine Angew. Math. 622 (2008), 95--119.


\bibitem[GP]{GP} V.A.~Galaktionov and L.A.~Peletier, {\em Asymptotic behaviour
near finite-time extinction for the fast diffusion equation,} Arch. Rat.
Mech. Anal. 139 (1997), 83--98.

\bibitem[Hs1]{Hs1} S.Y.~Hsu, {\em Singular limit and exact decay rate of a nonlinear elliptic
equation}, Nonlinear Analysis TMA 75 (2012), no. 7, 3443-3455.

\bibitem[Hs2]{Hs2} S.Y.~Hsu, Existence and asymptotic behaviour of solutions of the very fast diffusion, Manuscripta Math. 140 (2013), no. 3--4, 441--460. 

\bibitem[HK]{HK} K.M.~Hui and Sunghoon Kim, {\em Extinction profile of the logarithmic
diffusion equation}, Manuscripta Math. 143 (2014), no. 3-4, 491--524.

\bibitem[P]{P} L.A.~Peletier, {\em The porous medium equation in
Applications of Nonlinear Analysis in the Physical Sciences},
H.~Amann, N.~Bazley, K.~Kirchgassner editors, Pitman, Boston, 1981.

\bibitem[PS]{PS} M.~del Pino and M.~S\'aez, {\em On the extinction profile
for solutions of $u_t=\Delta u^{(n-2)/(N+2)}$}, Indiana Univ. Math. J. 50
(2001), no. 1, 611--628.

\bibitem[V]{V} J.L.~Vazquez, {\em Smoothing and decay estimates for nonlinear
diffusion equations}, Oxford Lecture Series in Mathematics and its Applications
33, Oxford University Press, Oxford, 2006.

\end{thebibliography}
\end{document}